\documentclass[reqno]{amsart}
\usepackage{amssymb,url}
\usepackage{hyperref}
\date{\bf May 01, 2008}
\theoremstyle{plain}
\newtheorem{theorem}{Theorem}
\newtheorem{lemma}{Lemma}

\theoremstyle{remark}

\DeclareMathOperator{\Rset}{\mathbb R}

\allowdisplaybreaks

\begin{document}

\title[...boundary value problems via topological methods...]
{Existence results for quasilinear elliptic boundary value problems via topological methods}

\author[Q.-A. Ng\^{o}]{Qu\^{o}\hspace{-0.5ex}\llap{\raise 1ex\hbox{\'{}}}\hspace{0.5ex}c Anh Ng\^{o}}
\address[Q.-A. Ng\^{o}]{Department of Mathematics,\\
College of Science, Vi\^{e}t Nam National University\\
H\`{a} N\^{o}i, Vi\^{e}t Nam}
\email{\href{mailto: Q.-A. Ng\^{o} <bookworm\_vn@yahoo.com>}{bookworm\_vn@yahoo.com}}




\begin{abstract}
In this paper, existence and localization results of $C^1$-solutions to elliptic Dirichlet boundary value problems are established. The approach is based on the nonlinear alternative of Leray-Schauder.
\end{abstract}

\thanks{This paper was typeset using \AmS-\LaTeX}

\maketitle

\section{Introduction}

In this paper, we consider the boundary value problem
\begin{equation} \label{eq1}
\begin{gathered}
-\Delta_p u = f(x, u),\quad \mbox{in }\Omega,\\
 u=0,\quad \mbox{on }\partial\Omega,
 \end{gathered}
 \end{equation}
where $\Omega \subset \Rset^N$ ($N \geqq 1$) is a nonempty bounded open set with smooth boundary $\partial \Omega$ and $f : \Omega \times \Rset \to \Rset$ is a continuous function. We seek $C^1$-solutions, i.e. function $u \in C^1(\overline \Omega)$ which satisfy \eqref{eq1} in the sense of distributions. 

In recent years, many authors have studied the existence of solutions for problem \eqref{eq1} from several points of view and with different approaches (see,for example, \cite{A, AR, CCN, CTY}). For instance, Afrouzi and Rasouli \cite{AR} ensure the existence of solutions for special types of nonlinearities, by using the method of sub- and supersolutions. 

Existence and multiplicity results for problem \eqref{eq1} are also presented by Anello \cite{A}, where $f$ admits the decomposition $f = g + h$ with $g$ and $h$ two Carath\'{e}odory functions having no growth conditions with respect to the second variable. His approach is variational and mainly based on a critical point theorem by B. Ricceri. 

In \cite{CCN}, Castro, Cossio and Neuberger apply the minmax principle to obtain sign-changing solutions for superlinear and asymptotically linear Dirichlet problems. 

A novel variational approach is presented by Costa, Tehrani and Yang \cite{CTY} to the question of existence and multiplicity of positive solutions to problem \eqref{eq1}, where they consider both the sublinear and superlinear cases. Another useful method for the investigation of solutions to semilinear problems is based on the Leray-Schauder continuation principle, or equivalently, on Schaefer's fixed point theorem. For example, in \cite{GT} this method is used for solutions in H\"{o}lder spaces, while in \cite{OP}, solutions are found in Sobolev spaces.

In this paper, we present new existence and localization results for $C^1$-solutions to problem \eqref{eq1}, under suitable conditions on the nonlinearity $f$. No growth conditions of "subcritical exponent" type are required. Our approach is based on regularity results for the solutions of linear Dirichlet problems and again on the nonlinear alternative of Leray-Schauder (see \cite{DG, Z}). We also notice that our present paper is motivated by the paper \cite{MP} where the same results are obtained for semilinear elliptic boundary value problems.

Our approach is mainly based on the following well-known theorem.

\begin{theorem}\label{dl1}
Let $B[0, r]$ denote the closed ball in a Banach space $E$ with radius $r$, and let $T : B[0, r] \to E$ be a compact operator. Then either
\begin{enumerate}
  \item[(i)] the equation $\lambda Tu = u$ has a solution in $B[0, r]$ for $\lambda = 1$, or
  \item[(ii)] there exists an element $u \in E$ with $\| u \| = r$ satisfying $\lambda Tu = u$ for some $0 < \lambda < 1$.
\end{enumerate}
\end{theorem}

It is worth noticing that contrary to most papers in the literature where the Leray-Schauder principle is used together with the a priori bounds technique, in the proof of our main result, Theorem \ref{dl2}, no a priori bounds of solutions of \eqref{eq3} are established. In addition, Theorem \ref{dl2} not only that guarantees the existence of a solution, but also gives information about its localization. This is derived from a very general growth condition, inequality \eqref{eq3}, which in particular contains both sublinear and superlinear cases without any restriction of exponent.

\section{Main results}

Here and in the sequel $E$ will denote the space 
\[
C_0 \left( {\overline \Omega  } \right) = \left\{ {u \in C\left( {\overline \Omega  } \right):u = 0  {\mbox{ on }} \partial \Omega } \right\}
\]
endowed with the sup-norm
\[
\left\| u \right\|_0  = \mathop {\sup }\limits_{x \in \overline \Omega  } \left| {u\left( x \right)} \right|.
\]
Also by $C^1_0 (\Omega)$ we mean the space $C^1(\overline \Omega)\cap C_0(\overline \Omega)$. We start with an existence and localization principle for \eqref{eq1}

\begin{theorem}\label{dl2}
Assume that there is a constant $r > 0$, independent of $\lambda > 0$, with
\begin{equation}\label{eq2}
\|u\|_0 \ne r,
\end{equation}
for any solution $u \in C^1_0(\overline \Omega)$ to
\begin{equation} \label{eq3}
\begin{gathered}
-\Delta_p u =\lambda f(x, u),\quad \mbox{in }\Omega,\\
 u=0,\quad \mbox{on }\partial\Omega,
 \end{gathered}
 \end{equation}
and for each $\lambda \in (0, 1)$. Then the boundary value problem \eqref{eq1} has at least one solution $u \in C^1_0(\overline \Omega)$ with $\|u\|_0 \leqq r$.
\end{theorem}

In order to prove Theorem \ref{dl2}, we firstly recall a well-known property of the operator $-\Delta_p$.

\begin{lemma}[See \cite{AC}, Lemma 1.1]\label{bd1}
Let $\Omega \subset \Rset^N$ be a bounded domain of class $C^{1, \beta}$ for some $\beta \in (0, 1)$ and $g \in L^\infty( \Omega)$. Then the problem
\begin{equation} \label{eq4}
\begin{gathered}
\int_\Omega {| \nabla u|^{p-2} \nabla u \nabla \varphi dx} =\int_\Omega {g  \varphi dx} ,\quad \mbox{ for all }\varphi \in C_0^\infty(\Omega),\\
 u \in W_0^{1,p}(\Omega),\quad p>1,
 \end{gathered}
 \end{equation}
has a unique solution $u \in C^1_0(\overline \Omega )$. Moreover, if we define the operator $K: L^\infty(\Omega) \to C^1_0(\overline \Omega): g \mapsto u$ where $u$ is the unique solution of (6), then $K$ is continuous, compact and order-preserving.
\end{lemma}

\begin{proof}[PROOF OF THEOREM \ref{dl2}]
According to Lemma \ref{bd1}, the operator $ (-\Delta_p)^{-1}$ from $L^\infty (\Omega)$ to $C^1_0 (\overline \Omega)$ is well-defined, continuous, compact and order-preserving. We shall apply Theorem \ref{dl1} to $E = C_0(\overline \Omega)$ and to the operator $T : C_0(\overline \Omega) \to C_0(\overline \Omega)$, with $Tu = (-\Delta_p)^{-1}Fu$, where $F : C(\overline \Omega) \to C(\overline \Omega)$ is given by $(Fu) (x) = f(x, u (x))$. Notice that, On the other hand, it is clear that the fixed points of $T$ are the solutions of problem \eqref{eq1}. Now the conclusion follows from Theorem \ref{dl1} since condition (ii) is excluded by hypothesis.
\end{proof}

Theorem 2.1 immediately yields the following existence and localization result.

\begin{theorem}\label{dl3}
Assume that there exist nonnegative continuous functions $\alpha$, $\beta$ and a continuous nondecreasing function $\psi : \Rset_+ \to \Rset_+$ such that
\begin{equation}\label{eq5}
\left| {f\left( {x,u} \right)} \right| \leqq \alpha \left( x \right)\psi \left( {\left| u \right|} \right) + \beta \left( x \right), \quad \forall \left( {x,u} \right) \in \Omega  \times \Rset.
\end{equation}
Suppose in addition that there exists a real number $r>0$ such that
\begin{equation}\label{eq6}
r \geqq \left\| {\left( { - \Delta _p } \right)^{ - 1} \alpha } \right\|_0 \psi \left( r \right) + \left\| {\left( { - \Delta _p } \right)^{ - 1} \beta } \right\|_0 .
\end{equation}
Then the boundary value problem \eqref{eq1} has at least one solution in $C_0^1(\overline \Omega)$ with $\|u\|_0 \leqq r$.
\end{theorem}

\begin{proof}
In order to apply Theorem \ref{dl2}, we have to show that condition \eqref{eq2} holds true for all solutions to \eqref{eq3}. Assume $u$ is any solution of \eqref{eq3} for some $\lambda \in (0,1)$ with $\|u\|_0=r$. Then
\[
u = \lambda Tu = \lambda \left( { - \Delta _p } \right)^{ - 1} Fu.
\]
Futhermore, for all $x \in \overline \Omega$, we have
\begin{align*}
  \left| {u\left( x \right)} \right| &= \lambda \left| {\left( { - \Delta _p } \right)^{ - 1} Fu\left( x \right)} \right| \hfill \\
 &  \leqq \lambda \left| {\left( { - \Delta _p } \right)^{ - 1} \left( {\alpha \left( x \right)\psi \left( {\left| {u\left( x \right)} \right|} \right) + \beta \left( x \right)} \right)} \right| \hfill \\
 &  \leqq \lambda \left( {\left\| {\left( { - \Delta _p } \right)^{ - 1} \alpha } \right\|_0 \psi \left( {\left\| u \right\|_0 } \right) + \left\| {\left( { - \Delta _p } \right)^{ - 1} \beta } \right\|_0 } \right) \hfill \\
 &  \leqq \lambda \left( {\left\| {\left( { - \Delta _p } \right)^{ - 1} \alpha } \right\|_0 \psi \left( r \right) + \left\| {\left( { - \Delta _p } \right)^{ - 1} \beta } \right\|_0 } \right).
\end{align*} 
Taking the supermum in the above inequality, we obtain
\[
\left\| u \right\|_0  \leqq \lambda \left( {\left\| {\left( { - \Delta _p } \right)^{ - 1} \alpha } \right\|_0 \psi \left( r \right) + \left\| {\left( { - \Delta _p } \right)^{ - 1} \beta } \right\|_0 } \right).
\]
Therefore $r \leqq \lambda r < r$ since $\lambda \in (0,1)$ and $\|u\|_0 \leqq r$. This is a contradiction.
\end{proof}

\section*{Acknowledgments}

This manuscript has NOT been SUBMITTED/ACCEPTED/REJECTED for publication before. This is the FIRST submission.


\begin{thebibliography}{999999}

\bibitem[A]{A}
\textsc{G. Anello},
Existence of solutions for a perturbed Dirichlet problem without growth conditions, 
{\em J. Math. Anal. Appl.} {\bf 330}(2), 1169-1178.

\bibitem[AC]{AC}
\textsc{C. Azizieh and Ph. Cl\'{e}ment},
A priori estimates and continuation methods for positive solutions of p-Laplace equations,
{\em J. Diff Eqns}, {\bf 179} (2002), 213-245.

\bibitem[AR]{AR}
\textsc{G.A. Afrouzi and S.H. Rasouli},
On positive solutions for some nonlinear semipositone elliptic boundary value problems
{\em Nonlinear Analysis: Modelling and Control}, {\bf 11} (4), (2006), 323-329.

\bibitem[CCN]{CCN}
\textsc{A. Castro, J. Cossio and J.M. Neuberger},
A minmax principle, index of the critical point, and existence of sign-changing solutions to elliptic boundary value problems, 
{\em Electron. J. Diff Eqns}, 1998, no. 2, 18 pp.

\bibitem[CTY]{CTY}
\textsc{D.G. Costa, H. Tehrani and J.J. Yang},
On a variational approach to existence and multiplicity results for semipositone problems , 
{\em Electron. J. Diff Eqns}, 2006, no. 11, 10 pp.

\bibitem[DG]{DG}
\textsc{J. Dugundji and A. Granas},
{\em Fixed Point Theory, Monographie Math.}, Warsaw, 1982.

\bibitem[GT]{GT}
\textsc{D. Gilbarg and N. Trudinger},
{\em Elliptic Partial Differential Equations of Second Order}, Springer, Berlin, 1983.

\bibitem[MP]{MP}
\textsc{T. Moussaoui and R. Precup},
Existence results for semilinear elliptic boundary value problems via topological methods, 
{\em Applied Mathematics Letters (2008)}, doi:10.1016/j.aml.2008.03.002.

\bibitem[OP]{OP}
\textsc{D. O'Regan and R. Precup},
{\em Theorems of Leray-Schauder Type and Applications}, Gordon and Breach, Amsterdam, 2001.

\bibitem[Z]{Z}
\textsc{E. Zeidler},
{\em Nonlinear Functional Analysis : Part I}, Springer-Verlag, New York, 1985.

\end{thebibliography}
\end{document}